\documentclass[12pt]{amsart}
\usepackage[T1]{fontenc}
\usepackage[utf8]{inputenc}
\usepackage{graphicx}
\usepackage{amssymb}
\usepackage{amsmath}
\usepackage{amsthm,amsfonts,bbm}
\usepackage{amscd}
\usepackage{geometry}
\usepackage{mathtools}
\usepackage{epsfig,epstopdf}
\usepackage{mathrsfs}
\usepackage[backref,colorlinks=true,linkcolor=blue,citecolor=blue,urlcolor=blue]{hyperref}
\usepackage{color}
\usepackage{tikz}
\usepackage{fixme}

\theoremstyle{plain}
\newtheorem{thm}{Theorem}[section]
\newtheorem{coro}[thm]{Corollary}
\newtheorem{lemma}[thm]{Lemma}
\newtheorem{prop}[thm]{Proposition}

\theoremstyle{definition}

\newtheorem{rmk}[thm]{Remark}

\numberwithin{equation}{section}
\numberwithin{figure}{section}
\def\B{{\sf B}}
\def\d{\delta}
\def\deg{\operatorname{deg}}
\def\dn{\mathrm d}
\def\down{\mathrm{down}}
\def\ex{\operatorname{ex}}
\def \im{\operatorname{im}}
\def \lk{\operatorname{lk}}
\def\p{\partial}
\def\q{\mathfrak q}
\def\R{\mathbb R}

\def\sp{\operatorname{spex}}
\def\supp{\operatorname{supp}}
\def\un{\mathrm u}
\def\up{\mathrm{up}}

\begin{document}

\title[Spectral Radius of Simplicial Complexes]{Signless Laplacian Spectral Radius and Link Homology of Simplicial Complexes}

\author[Y.-Z. Fan]{Yi-Zheng Fan*}
\address{\small Center for Pure Mathematics, School of Mathematical Sciences,  Anhui University, Hefei 230601, P. R. China}
\email{fanyz@ahu.edu.cn}
\author[H.-Z. Zhang]{Huan-Zhi Zhang}
\address{\small School of Mathematical Sciences, Anhui University, Hefei 230601, P. R. China}
\email{zhanghz@stu.ahu.edu.cn}
\thanks{*The corresponding author.
Supported by National Natural Science Foundation of China (No. 12331012).
}

\subjclass[2020]{Primary 05E45, 15A18; Secondary 05C65, 05C35}

\keywords{Simplicial complex; signless Laplacian; spectral radius; link; homology; generalized wheel; generalized book}

\begin{abstract}
In this paper, we study the signless Laplacian spectral radius of pure simplicial complexes under local homological restrictions on links.
Let $K$ be a pure $r$-dimensional complex on $n$ vertices, $\q_{r-1}(K)$ be the spectral radius of the $(r-1)$-up signless Laplacian of $K$, and $\lk_K(\sigma)$ be the link of a face $\sigma$ in $K$.
We prove that if the homology $\widetilde H_t(\lk_K(\sigma), \R)=0$ for every face $\sigma\in K$ with $|\sigma|=r-t$, then
\[ \q_{r-1}(K)\le tn-(t-1)(r+1).\]
Moreover, if $K$ is $r$-down path connected and $n\ge r+2+\binom{r+1}{t}\binom{r}{t}$, equality holds if and only if $K \cong \Delta_{r+1-t} \star \Delta_{n-r-1+t}^{t}$,
where $\Delta_n$ denotes a simplex on $n$ vertices, $\Delta_n^{p}$ denotes the $(p-1)$-skeleton of $\Delta_n$, and $\star$ denotes the join of two complexes.
\end{abstract}

\maketitle

\section{Introduction}
\subsection{Classical and high-dimensional Tur\'an problems}
In extremal combinatorics, the classical Tur\'an problem seeks to determine the maximum number of edges in an $n$-vertex graph that does not contain a given forbidden subgraph $F$.
The Tur\'an number of a graph $F$, denoted by $\ex(n,F)$, is defined as the maximum number of edges in an $F$-free graph on $n$ vertices.
This area was pioneered by Mantel and Tur\'an \cite{Turan1941}, and it has since become a cornerstone of discrete mathematics.
A milestone in this field is the Erd\H{o}s-Stone-Simonovits theorem \cite{ES1966, Simon1968}, which determines the asymptotic behavior of the Tur\'an number for any non-bipartite graph $F$ based on its chromatic number.

The classical Tur\'an problem naturally extends to hypergraphs and higher-dimensional structures.
An \emph{abstract simplicial complex}, or simply a \emph{complex}, is a collection $K$ of finite sets closed under inclusion.
An element of cardinality $r+1$ is called an \emph{$r$-face} or an \emph{$r$-simplex} with dimension $r$.
The faces that are maximal under inclusion are called \emph{facets}.
An $r$-dimensional complex is \emph{pure} if all its facets have dimension $r$.
Let $\Delta_{n}$ denote an $(n-1)$-simplex, which can also be considered a simplicial complex on $n$ vertices with all possible subsets as faces, and let $\Delta_n^{r+1}$ denote the $r$-skeleton of $\Delta_n$, namely, the subcomplex consisting of all faces of $\Delta_n$ of dimension at most $r$.
Thus, a pure $r$-dimensional complex may be viewed as an $(r+1)$-uniform hypergraph together with all lower-dimensional faces.
As part of the program of high-dimensional combinatorics, Linial \cite{Lini2008, Lini2018, LL2016, LP2016} proposed the following problem:
Given an $r$-dimensional complex $K$, how many facets can an $r$-dimensional complex on $n$ vertices have if it contains no homeomorphic copy of $K$?

When $r=1$, this problem reduces to the classical Tur\'an problem for graphs.
It is thus a generalization of the Tur\'an problem to higher dimensions.
For an $r$--dimensional complex $K$, let $\ex_r(n,K)$ denote the maximum number of facets in an $r$-complex on $n$ vertices that contains no copy of $K$ as a subcomplex.
The higher-dimensional Tur\'an problem is already difficult in small dimensions.
A prime example is the $2$-dimensional complex $\Delta_4^{3}$.
Tur\'an \cite{Turan1941} posed a conjecture on the exact value of $\ex_2(n , \Delta_4^{3})$, which is still unknown to this day.
Long, Narayanan, and Yap proved the following general upper bound for homeomorphic Tur\'an problems.

\begin{thm}[Long-Narayanan-Yap \cite{LNY2022}] \label{LNY2022}
For every $r\in\mathbb N$, there exists a constant $\lambda_r\ge r^{-2r^2}$ such that, for every pure $r$-complex $K$, there is a constant $C=C(K)$ satisfying
$$
	\ex_r(n,K)\le C n^{r+1-\lambda_r}.
$$
\end{thm}

The existence of such universal exponents $\lambda_r$ as in Theorem \ref{LNY2022} was previously known only for $r = 1$ and $r = 2$; that the optimal value of $\lambda_1$ is $1$ is a classical result of Mader \cite{Mader1967}, and that $\lambda_2 \geq 1/5$ was shown recently by Keevash, Long, Narayanan, and Scott \cite{KLNS2021}.
It was conjectured that the optimal value of $\lambda_2$ is $1/2$, which remains open.

Let $\mathcal{F}$ be a family of $r$-dimensional complexes, and let $\ex_r(n,\mathcal{F})$ denote the maximum number of facets in an $r$-dimensional complex on $n$ vertices that contains no member of $\mathcal{F}$ as a subcomplex.
The case of two-dimensional surfaces has also been studied intensively.
A classical result of Brown, Erdős, and Sós \cite{SEB1973} first established that $\ex_2(n, S^2) = \Theta(n^{5/2}) $, where $S^r$ denotes an $r$-dimensional sphere.
Kupavskii, Polyanskii, Tomon, and Zakharov  \cite{KPTZ2022} extended this order of magnitude to triangulations of all closed orientable surfaces.
Sankar  \cite{San2024} subsequently obtained the corresponding result for non-orientable surfaces.
These results illustrate the role of topological forbidden configurations in extremal combinatorics.	
More recently, Newman and Pavelka \cite{NP2024} provided a conditional lower bound and also an upper bound for $\ex_r(n,S^r)$.
A systematic investigation of non-pure extremal problems was recently initiated by Conlon, Piga, and Sch\"ulke \cite{CPS2023} and was further studied by Axenovich et al. \cite{AGLP2025}.
For more related work, see the survey  \cite{K2011, Niki2011}.

\subsection{Spectral Tur\'an problem for complexes}
In recent decades, a spectral analog of the Tur\'an problem has attracted significant attention.
Usually, the spectral radius of a graph $G$ refers to  the spectral radius of the adjacency matrix $A(G)$ of $G$, denoted by $\lambda(G)$.
Let $\sp(n,F)$ denote the maximum spectral radius among all $n$-vertex graphs that do not contain $F$ as a subgraph.
Initiated by Nikiforov \cite{Niki2007}, the spectral Tur\'an problem originally focused on determining the maximum adjacency spectral radius of an $F$-free graph on $n$ vertices.
Nikiforov established several foundational results, including the spectral versions of Tur\'an's theorem and the Erd\H{o}s-Stone-Simonovits theorem \cite{Niki2007, Niki2009}.
The spectral Tur\'an problem is closely related to  the Tur\'an problem by the following connection: $\lambda(G) \ge {2e(G)}/{n}$ for a graph $G$ on $n$ vertices, which implies that $ \ex(n,G) \le {n \cdot \sp(n,G)}/{2}.$
Subsequently, the study of spectral Tur\'an problems was naturally extended to the Laplacian and signless Laplacian matrices.
For surveys, see \cite{LB2023,LLF2022}.

It is therefore natural to ask for spectral analogs of these high-dimensional
Tur\'an problems.
The spectral theory of combinatorial Laplacians on simplicial complexes goes back to Eckmann's discrete Hodge theorem \cite{E1944}.
Horak and Jost developed a general framework for Laplace operators on simplicial complexes \cite{HJ2013a}.
The spectral Tur\'an problem for simplicial complexes via the signless Laplacian was initiated by Fan, She, and Zhang \cite{FSZ2025}.
For a pure $r$-dimensional complex $K$, let $\q_{r-1}(K)$ denote the spectral radius of its $(r-1)$-up signless Laplacian.
Fan, She, and Zhang \cite{FSZ2025} determined the maximum value of $\q_{r-1}(K)$ for complexes without $r$-dimensional holes.
In a subsequent paper, She, Fan, and Song \cite{SFS2026} established an asymptotic formula for the maximum signless Laplacian spectral radius of $r$-dimensional pure simplicial complexes on $n$ vertices with prescribed second Betti number $t$ for $1 \le t \le n-3$ and determined the extremal complexes when $t$ equals $1$ or $2$.
Recently, Zhang and Fan \cite{ZF2026} provided the value of $\q_{r-1}(K)$ for complexes without $r$-dimensional wheels and determined the extremal complexes.

In addition to the spectral radius, the smallest nonzero eigenvalue of the complex Laplacian, referred to as the spectral gap, has also been widely investigated.
Lew \cite{AL2020} gave the lower bound of the spectral gap of the Laplacian for $r$-dimensional complexes without $d$-dimensional holes and proposed a conjecture on the extremal case of the $r$-dimensional Laplacian.
Recently, Zhan, Huang, and Lin \cite{ZHL2026} resolved this conjecture.

\subsection{Link homology viewpoint}
The homology of links plays an important role in the study of simplicial complexes.
A classical source of this viewpoint is Cohen-Macaulay theory.
Let $K$ be a simplicial complex.
For a face $\sigma\in K$, the \emph{link} of $\sigma$ in $K$ is
\[
\lk_K(\sigma) =\{\tau\in K:\tau\cap \sigma=\emptyset,\ \tau\cup\sigma\in K\}.
\]
By Reisner's criterion \cite{Reisner1976}, a simplicial complex $K$ is \emph{Cohen-Macaulay} over a field $\mathbb F$ if and only if, for every face
$\sigma\in K$,
\[
 \widetilde H_i(\lk_K(\sigma), \mathbb F)=0, ~ \text{for all } i<\dim \lk_K(\sigma).
\]
That is, a simplicial complex is Cohen-Macaulay over  $\mathbb F$ if and only if the homology of each face’s link vanishes below its top dimension.
This criterion is one of the basic connections between the topology of
simplicial complexes and Stanley-Reisner theory; see survey \cite{FJJ2014}.

The condition studied in the present paper is inspired by this
link-homology viewpoint.
Let $K$ be a pure $r$-dimensional simplicial complex.
For a fixed integer $t$ with $1\le t\le r$, we impose the following
fixed-codimension top-link homology condition:
\[
\widetilde H_t(\lk_K(\sigma), \R)=0, ~ \text{for every face~} \sigma\in K ~ \text{with~} |\sigma|=r-t.
\]
Since $K$ is pure $r$-dimensional, each such link has dimension exactly $t$.
Thus, the condition asks for the vanishing of the top-dimensional homology of these fixed-codimension links.
This is the bridge from local topology to spectral extremal theory.

\subsection{Main result and proof technique}
We first introduce two classes of simplicial complexes constructed by join.
The \emph{join} of two complexes $X$ and $Y$ on disjoint vertex sets is the complex
\[ X\star Y=\{\sigma\cup\tau:\sigma\in X,\ \tau\in Y\}.\]
For $1\le t\le r$, define 
$\B_n(r,t) = \Delta_{r+1-t}\star \Delta_{n-r-1+t}^{t}$.
We call $\B_n(r,t)$ the \emph{generalized $(r,t)$-book} on $n$ vertices.
Equivalently, $\B_n(r,t)$ is the pure $r$-dimensional complex whose every $r$-face contains a fixed core of size $r+1-t$ and whose remaining $t$ vertices are chosen arbitrarily from the other $n-r-1+t$ vertices.
For example, when $r=2$ and $t=1$, $\B_n(2,1)$ is the usual book of triangles sharing a common edge, and the $1$-skeleton of $\B_n(2,1)$ is exactly the book graph.
When $r=2$ and $t=2$, $\B_n(2,2)$ is the cone over the complete graph on $n-1$ vertices.

For $1\le t\le r$, a \emph{generalized $(r,t)$-wheel} is an $r$-dimensional complex of the form $\Delta_{r-t} \star H$,
where $H$ is a $t$-dimensional basic hole on a vertex set disjoint from $\Delta_{r-t}$; for the notion of  basic holes; see Section \ref{sec-2.3}.

Our main theorem is as follows.

\begin{thm}\label{main}
Let $1\le t\le r$, and let $K$ be a pure $r$-dimensional complex on $n$ vertices.
Suppose that $\widetilde H_t(\lk_K(\sigma), \R)=0$ for every face $\sigma\in K$ with $|\sigma|=r-t$.
Then
\[	\q_{r-1}(K)\le tn-(t-1)(r+1).	\]
Moreover, if $K$ is $r$-down path connected and $n\ge r+2+\binom{r+1}{t}\binom{r}{t}$, then equality holds if and only if $K\cong \B_n(r,t)$.
\end{thm}

There is a combinatorial interpretation of the condition on the link homology in Theorem \ref{main}.
By Lemma \ref{link}, it is equivalent to $K$ containing no generalized $(r,t)$-wheels.
So we have the following equivalent statement of Theorem \ref{main}.

\begin{coro}\label{wheel-free-main}
Let $1\le t\le r$, and let $K$ be a pure $r$-dimensional complex on $n$ vertices without generalized $(r,t)$-wheels.
Then
\[	\q_{r-1}(K)\le tn-(t-1)(r+1).	\]
Moreover, if $K$ is $r$-down path connected and
	$n\ge r+2+\binom{r+1}{t}\binom{r}{t}$,
then equality holds if and only if 	$K\cong \B_n(r,t)$.
\end{coro}

Corollary \ref{wheel-free-main} interpolates between two previously studied extremal signless Laplacian problems for simplicial complexes.
Indeed, when $t=r$, the core $\Delta_{r-t}$ is the simplex on the empty vertex set, and then the generalized  $(r,r)$-wheels are precisely $r$-dimensional basic holes.
Thus, the case $t=r$ recovers the $r$-hole-free setting studied by Fan, She, and Zhang \cite{FSZ2025}.
In this case, the extremal complex becomes $\B_n(r,r) = \Delta_1\star \Delta_{n-1}^{r}$.
At the other extreme, when $t=1$, a $1$-dimensional basic hole is precisely a cycle, and hence the generalized $(r,1)$-wheels are the usual $r$-dimensional wheels introduced by S\'os, Erd\H{o}s and Brown \cite{SEB1973}.
Hence, the case $t=1$ corresponds to the wheel-free setting studied by Zhang and Fan \cite{ZF2026}.
In this case, the extremal complex becomes $\B_n(r,1)=\Delta_r\star \Delta_{n-r}^{1}$.

Therefore, the present theorem does not merely reprove the hole-free and wheel-free results separately.
Rather, it places them in a single fixed-codimension link-homology framework.
For each $1\le t\le r$, we forbid top-dimensional homology in the links of faces of cardinality $r-t$, and this local homological condition is converted into a uniform local exchange bound.
This exchange bound leads to the sharp estimate
\begin{equation}\label{upspecbound}
\q_{r-1}(K)\le tn-(t-1)(r+1).
\end{equation}
The equality case is then governed by a rigidity phenomenon: $t$-neighbor uniformity; see Section \ref{sec-2.1}.

The main difficulty is showing that this local exchange uniformity has a global consequence.
Under the $r$-down path connectedness, we prove that for sufficiently large $n$,  $K$ has a global common core of size $r+1-t$, and $K$ is therefore isomorphic to $\B_n(r,t)=\Delta_{r+1-t}\star \Delta_{n-r-1+t}^{t}$.
The technique used here involves co-graph complexes and local exchange operations.

By Corollary \ref{wheel-free-main}, we provide an upper bound for the facet number of a pure $r$-dimensional complex without generalized $(r,t)$-wheels, where $f_i(K)$ denotes the number of $i$-faces in a complex $K$.

\begin{coro}\label{face-bound}
Let $1\le t\le r$, and let $K$ be a pure $r$-dimensional complex on $n$ vertices without generalized $(r,t)$-wheels.
Then
	\[ 	f_r(K)	\le	\frac{tn-(t-1)(r+1)}{(r+1)^2}\, f_{r-1}(K).	\]
In particular,
	\[	f_r(K)	\le	\frac{tn-(t-1)(r+1)}{(r+1)^2}\binom{n}{r}.	\]
\end{coro}

\begin{proof}
This follows from Corollary \ref{wheel-free-main} and the Rayleigh quotient estimate \eqref{raylei-face}.
\end{proof}

The $t$-neighbor uniformity was used to characterize the equality case in \eqref{upspecbound}, where $1 \le t \le r$.
In the following, we use $(r+1)$-neighbor uniformity to characterize the equality case of the spectral upper bound for the Cohen-Macaulay pure $r$-dimensional complexes.

\begin{thm}\label{cmbound}
Let $K$ be a Cohen-Macaulay pure $r$-dimensional complex over a field $\mathbb F$ on $n$ vertices. Then
	\[	\q_{r-1}(K)\le (r+1)(n-r).	\]
Moreover, equality holds if and only if $K=\Delta_n^{r+1}$.
\end{thm}

For comparison, we also provide the following absolute upper bound, which holds for all pure $r$-dimensional complexes without any homological assumptions.

\begin{prop}\label{absolute-bound}
Let $K$ be a pure $r$-dimensional complex on $n$ vertices.
Then
	\[ 	\q_{r-1}(K)\le (r+1)(n-r).	\]
Moreover, equality holds if and only if $K=\Delta_n^{r+1}$.
\end{prop}

\begin{rmk}
Since Proposition \ref{absolute-bound} holds for all pure $r$-dimensional complexes, it applies in particular to Cohen-Macaulay pure complexes over any field.
In this class, the equality case is still $\Delta_n^{(r)}$, which is Cohen-Macaulay over every field.
\end{rmk}

The paper is organized as follows.
In Section 2, we collect the necessary notation on homology, basic holes, and the signless Laplacian.
In Section 3, we give a combinatorial interpretation of the condition on the link homology, that is, $\widetilde H_t(\lk_K(\sigma), \R)=0$ for every $(r-t-1)$-face $\sigma \in K$ if and only if $K$ contains no generalized $(r,t)$-wheels, and derive a spectral upper bound for an $r$-dimensional complex $K$ without generalized $(r,t)$-wheels.
If $K$ is $r$-down path connected, then the equality in the spectral upper bound holds if and only if $K$ is $t$-neighbor uniform.
In addition, Proposition \ref{absolute-bound} is proved here, which implies Theorem \ref{cmbound}.
In Section 4, under $r$-down path connectedness, we prove that for sufficiently large $n$, the extremal complex $K$ attaining the spectral upper bound contains a global common core of size $r+1-t$ and hence is isomorphic to $\B_n(r,t)$.
We finally prove Corollary \ref{wheel-free-main}, an equivalent statement of Theorem \ref{main}.

\section{Preliminaries}
\subsection{Simplicial complexes and homology}\label{sec-2.1}
Throughout the paper, all simplicial complexes are finite.
Let $K$ be a simplicial complex with vertex set $V(K)$.
For $X\subseteq V(K)$, the \emph{induced subcomplex} $K[X]$ is the complex whose simplices are all faces of $K$ contained in $X$.
The \emph{dimension} of an $i$-face is $i$, and the \emph{dimension} of $K$ is the maximum dimension of all faces of $K$.
We allow $\emptyset \in K$, whose dimension is $-1$.
Let $S_i(K)$ denote the set of all $i$-faces of $K$, where  $S_{-1}(K)=\{\emptyset\}$.
The \emph{$p$-skeleton} of $K$, denoted by $K^{(p)}$, is the subcomplex consisting of all faces of dimension at most $p$.
For a finite vertex set $X$, let $\Delta_X$ denote the full simplex on $X$.
If $|X|=n$, we also write $\Delta_n$ for a simplex on $n$ vertices.

For two $i$-faces of $K$ that are faces of an $(i+1)$-face, we call them \emph{up neighbors}.
An \emph{$i$-up path} is a sequence of $i$-faces $F_1,F_2,\ldots,F_m$ such that $F_j$ and $F_{j+1}$ are up neighbors for all $j$.
The complex $K$ is called \emph{$i$-up path connected} if any two $i$-faces are connected by an $i$-up path.
For an $i$-face $F$, let $N^{\un}(F)$ be the set of up neighbors of $F$.
The \emph{degree} of $F$, denoted by $\deg(F)$, is the cardinality of $N^{\un}(F)$.

Similarly, for two $i$-faces of $K$ that share an $(i-1)$-face, we refer to them as \emph{down neighbors}.
An \emph{$i$-down path} is a sequence of $i$-faces $F_1,F_2,\ldots,F_m$ such that $F_j$ and $F_{j+1}$ are down neighbors for all $j$.
The complex $K$ is called \emph{$i$-down path connected} if any two $i$-faces are connected by an $i$-down path.
For an $i$-face $F$ and a vertex $u\in V(K)\setminus F$, let $N^\dn(F)$ be the set of down neighbors of $F$, and let
$$ N^\dn(F,u):=\{F'\in N^\dn(F):u\in F'\}.$$
For $1 \le t \le r+1$, a pure $r$-dimensional complex $K$ is called \emph{$t$-neighbor uniform} if
$$|N^\dn(F,u)|=t$$
for every $F\in S_r(K)$ and every $u\in V(K)\setminus F$.

A face $F$ of $K$ is \emph{oriented} if we choose an ordering of its vertices and write $[F]$.
Two orderings of the vertices are said to determine the same orientation if there is an even permutation transforming one ordering into the other.
If the permutation is odd, then the orientations are opposite.

After choosing one orientation for each $i$-face, the \emph{$i$-th chain group} $C_i(K,\R)$ of $K$ with coefficients in $\R$ is the vector space over $\R$ with basis
$\{[F]:F\in S_i(K)\}$.
In particular, set $C_{-1}(K,\R)=\R$.
The \emph{$i$-th boundary map} $\partial_i:C_i(K,\R)\to C_{i-1}(K,\R)$ is defined by
\[
\partial_i[v_0,\ldots,v_i] = \sum_{j=0}^i(-1)^j[v_0,\ldots,\widehat v_j,\ldots,v_i],
\]
where $\widehat v_j$ indicates that the vertex $v_j$ is omitted.
These boundary maps satisfy $\partial_i\partial_{i+1}=0$, and hence they form the chain complex
\[
\cdots \xrightarrow{\partial_{i+1}} C_i(K,\R) \xrightarrow{\partial_i} C_{i-1}(K,\R) \to\cdots \xrightarrow{\partial_1} C_0(K,\R) \xrightarrow{\partial_0} C_{-1}(K,\R) \to 0.
\]

The kernel of $\partial_i$, denoted by $Z_i(K,\R):=\ker\partial_i$, is called the \emph{space of $i$-cycles}.
The image of $\partial_{i+1}$, denoted by $B_i(K,\R):=\im\partial_{i+1}$, is called the \emph{space of $i$-boundaries}.
Since $\partial_i\partial_{i+1}=0$, we have $B_i(K,\R)\subseteq Z_i(K,\R)$.
The quotient
\[
\widetilde H_i(K,\R)=Z_i(K,\R)/B_i(K,\R)
\]
is called the \emph{$i$-th reduced homology group} of $K$ over $\R$.
Its dimension is the \emph{$i$-th Betti number} of $K$, denoted by $\beta_i(K,\R)$.
When $i\ge 1$, the $i$-th reduced homology agrees with ordinary $i$-th homology.

\subsection{Basic holes}\label{sec-2.3}
We first identify the minimal top-dimensional homological obstructions arising for links.
Let $H$ be a pure $t$-dimensional complex.
For a $t$-chain $z=\sum_{F\in S_t(H)}a_F [F]\in C_t(H,\R)$, define
$\supp(z)=\{F\in S_t(H):a_F\ne0\}$.
A nonzero $t$-cycle $z\in Z_t(H,\R)$ is called \emph{support-minimal} if there is no nonzero
$t$-cycle $z'\in Z_t(H,\R)$ such that $\supp(z')\subsetneq\supp(z)$.

A pure $t$-dimensional complex $H$ is called a \emph{$t$-dimensional basic hole}
over $\R$ if there exists a support-minimal nonzero $t$-cycle
$z\in Z_t(H,\R)$ such that $S_t(H)=\supp(z)$.
Throughout the paper, all basic holes are understood to be over $\R$.

This support formulation is equivalent to the usual homological description:
$\beta_t(H,\R)=1$ and $\beta_t(H-F,\R)=0$ for every $F\in S_t(H)$, where
$H-F$ denotes the complex obtained from $H$ by deleting the $t$-face $F$.
Indeed, if $S_t(H)=\supp(z)$ for a support-minimal nonzero $t$-cycle $z$, then
by the minimality of support, it is easy to show that $Z_t(H,\R)$ is one-dimensional.
Since $H$ is $t$-dimensional, $B_t(H,\R)=0$, and hence $\beta_t(H,\R)=1$.
The same minimality also gives $\beta_t(H-F,\R)=0$ for every $F\in S_t(H)$.

\begin{lemma}\label{basichole}
Let $L$ be a pure $t$-dimensional complex, where $t\ge1$.
If $\widetilde H_t(L,\R)\ne0$, then $L$ contains a $t$-dimensional basic hole as a subcomplex.
\end{lemma}

\begin{proof}
Since $\dim L = t$, we have $B_t(L,\R)=0$.
Hence, $\widetilde H_t(L,\R)\ne0$ implies that $Z_t(L,\R)$ contains a nonzero $t$-cycle.
Choose a support-minimal nonzero $t$-cycle $z\in Z_t(L,\R)$.
Let $H$ be the pure $t$-dimensional subcomplex of $L$ generated by the $t$-faces in $\supp(z)$ and all their subfaces.
Then $S_t(H)=\supp(z)$.
Moreover, $z$ remains support-minimal when regarded as a $t$-cycle of $H$ because every
$t$-cycle of $H$ is also a $t$-cycle of $L$.
Therefore, $H$ is a $t$-dimensional basic hole.
\end{proof}

We shall use two standard examples of basic holes.
First, $\Delta_{r+2}^{r+1}$ is the boundary complex of a
$(r+1)$-simplex.
Hence $\Delta_{r+2}^{r+1} \cong S^r$.
Second, let $\diamondsuit_{r+3}^{r+1}$ be the $r$-dimensional complex on vertex set $[r+3]$ with facets
$$
\left\{\{i\}\cup F:i\in\{r+2,r+3\},\ F\in\binom{[r+1]}r\right\}.
$$
Equivalently,
$\diamondsuit_{r+3}^{r+1}\cong \Delta_2^{1}\star \Delta_{r+1}^{r} \cong S^r$.
In both examples, the unique top-dimensional homology class is supported
on all facets; deleting any facet therefore kills the top-dimensional homology.
Hence, these complexes are basic holes.

Lemma \ref{basichole} shows that basic holes are the minimal top-dimensional homological obstructions.
In Section \ref{up-bound}, we use this observation to translate the condition on link homology in Theorem \ref{main} into the exclusion of generalized wheels.

\subsection{Signless Laplacian}\label{signless}
It is known that  the Laplacian of a simplicial complex is closely related to the discrete version of the Hodge theorem; see \cite{E1944,DR2002,HJ2013a}.
Here, we use the signless Laplacian in the spectral extremal problem of simplicial complexes.
The signless Laplacian of a simplicial complex was introduced in \cite{KO2020}, and has been systematically studied with many interesting applications \cite{KL2014,Lub2014}.
It was further extended to the signless $1$-Laplacian for investigating the combinatorial properties of a complex \cite{LZ2020}.
One can also refer to \cite{FWW2025,SWF2025} for the relationship between the spectral radius of Laplacians and signless Laplacians.

The signless Laplacian operator is obtained by omitting the signs in the usual boundary maps.
Let $D_i(K,\mathbb R)$ be the real vector space generated by the unoriented $i$-faces of $K$.
The \emph{$i$-th signless boundary map}
$|\p_i|: D_i(K,\R) \to D_{i-1}(K,\R)$ is defined by
$$ |\p_i|\{v_0,\ldots,v_i\}=\sum_{j=0}^i \{v_0,\ldots,\hat{v}_j, \ldots, v_i\}.$$
Let $D^i(K,\R)$ be the dual space of $D_i(K,\R)$, with dual basis
$\{F^*:F\in S_i(K)\}$.
The \emph{$i$-th signless coboundary map} $|\d_i|: D^i(K,\R) \to D^{i+1}(K,\R)$ is defined by $ |\d_i| f =f |\p_{i+1}|$ for each $f \in D^i(K,\R)$.
For each $i$, endow $D^i(K,\R)$ with a positive inner product that makes
the dual basis $\{F^*: F \in S_i(K)\}$ orthonormal.
Let $|\d_i|^*:D^{i+1}(K,\R) \to D^i(K,\R)$ be the adjoint of $|\d_i|$,
which satisfies
$$\langle |\d_i| f, g \rangle_{D^{i+1}(K,\R)}=\langle f, |\d_i|^* g \rangle_{D^{i}(K,\R)}$$
for all $f \in D^i(K,\R)$ and $g \in D^{i+1}(K,\R)$.
The \emph{$i$-th up signless Laplace operator} and \emph{$i$-th down signless Laplace operator} of $K$ are respectively defined by
\begin{equation}\label{SLaplace}
	Q_i^{\up}(K)=|\d_{i}^*| |\d_{i}|, \quad Q_i^{\down}(K)=|\d_{i-1}| |\d_{i-1}^*|.
\end{equation}

Equivalently, after choosing the natural basis of faces, these operators
are represented by unsigned incidence matrices.
Let $B_i$ be the matrix of $|\d_i|$.
Then $B_i$ is the unsigned incidence matrix between $i$-faces and $(i+1)$-faces, whose rows are indexed by $(i+1)$-faces and whose columns are
indexed by $i$-faces, with
\[
(B_i)_{\bar{F},F}
=\begin{cases}
	1, & \text{if } F\subset \bar{F},\\
	0, & \text{otherwise}.
\end{cases}
\]
Consequently, we have the following matrix forms:
\[
Q_i^{\up}(K)=B_i^{\top}B_i,
\quad
Q_{i}^{\down}(K)=B_{i-1}B_{i-1}^{\top}.
\]

For each $f \in D^i(K,\R)$, by definition, we have the following expressions (see \cite{FSZ2025, ZF2026}).
\begin{equation}\label{Q-up}
	(Q_i^{\up}(K)f)(F)=\deg (F)  f(F) + \sum_{F' \in N^\un(F)} f(F');
\end{equation}
\begin{equation}\label{Q-dn}
	(Q_i^{\down}(K)f)(F)=|F| f(F) + \sum_{F' \in N^\dn(F)} f(F').
\end{equation}

For a pure $r$-dimensional complex $K$, let $\q_{r-1}(K)$ denote the spectral radius of $Q_{r-1}^{\up}(K)$.
When $r=1$, the complex $K$ is a graph, and $Q_0^{\mathrm{up}}(K)$ is exactly the
signless Laplacian matrix of a graph.
Since $Q_{r-1}^\up$ and $Q_{r}^\down$ share the same nonzero eigenvalues, especially the spectral radius, we use $Q_{r-1}^\up$ or $Q_{r}^\down$ interchangeably in the subsequent discussion as convenient.

This operator provides a natural spectral replacement for the number of facets.
Let $f_r(K)$ be the size of the set $S_r(K)$.
Indeed, if $\mathbf 1$ denotes the all-one vector defined on $S_{r-1}(K)$,
then the Rayleigh quotient gives
\begin{equation}\label{raylei-face}
	\q_{r-1}(K)
	\geq
	\frac{\langle Q_{r-1}^{\up}(K)\mathbf 1,\mathbf 1\rangle}
	{\langle \mathbf 1,\mathbf 1\rangle}
	=
	\frac{(r+1)^2 f_r(K)}{f_{r-1}(K)}.
\end{equation}
In particular, since $f_{r-1}(K)\leq \binom{n}{r}$, an upper bound on
$\q_{r-1}(K)$ yields an upper bound on the number of facets $f_r(K)$.
Thus, the signless Laplacian spectral radius is directly connected with the
ordinary high-dimensional Tur\'an problem.

Although the homological condition in Theorem \ref{main} is defined using the ordinary signed boundary complex, the spectral operator in this paper is the signless Laplacian.
The bridge between these two objects is combinatorial.
The vanishing of the top homology of suitable links is first converted,
via Lemma \ref{link}, into the exclusion of generalized wheels.
This forbidden local configuration then gives the exchange bound in Lemma \ref{le-t}, which is precisely the input needed to control the row sums of $Q_r^{\down}(K)$.

\section{From link homology to the spectral upper bound}\label{up-bound}

\subsection{Link homology and generalized wheels}
We convert the condition  on link homology to the exclusion of generalized wheels and the bound of local down neighbor numbers.

\begin{lemma}\label{link}
Let $K$ be a pure $r$-dimensional complex, and let $1\le t\le r$.
Then $K$ contains no generalized $(r,t)$-wheels if and only if
	\[
	\widetilde H_t(\lk_K(\sigma), \R)=0
	\]
for every face $\sigma\in K$ with $|\sigma|=r-t$.
\end{lemma}

\begin{proof}
It is enough to prove the equivalent statement that $K$ contains generalized $(r,t)$-wheels if and only if there exists a face $\sigma\in K$ with $|\sigma|=r-t$ such that $\widetilde H_t(\lk_K(\sigma),\R)\ne 0$.
	
Suppose first that $K$ contains a generalized $(r,t)$-wheel.
Then there exists a face $\sigma\in K$ with $|\sigma|=r-t$ and a $t$-dimensional basic hole $H$ such that $\Delta_\sigma\star H\subseteq K$.
By the definition of the link, this implies $H\subseteq \lk_K(\sigma)$.
Since $H$ is a $t$-dimensional basic hole, it supports a nonzero $t$-cycle.
Regarded as a chain in $\lk_K(\sigma)$, this cycle is still nonzero.
Since $\dim\lk_K(\sigma)= t$, we have $B_t(\lk_K(\sigma),\R)=0$.
Therefore $\widetilde H_t(\lk_K(\sigma),\R)\ne0$.

Conversely, suppose that $\widetilde H_t(\lk_K(\sigma),\R)\ne 0$ for some face $\sigma\in K$ with $|\sigma|=r-t$.
Since $\dim\lk_K(\sigma)= t$, Lemma \ref{basichole} implies that $\lk_K(\sigma)$ contains a $t$-dimensional basic hole $H$ as a subcomplex.
Thus $H\subseteq \lk_K(\sigma)$.
By the definition of the link, $\Delta_\sigma\star H\subseteq K$.
Since $|\sigma|=r-t$, and $H$ is a $t$-dimensional basic hole, this is a generalized	$(r,t)$-wheel.
Thus $K$ contains generalized $(r,t)$-wheels.
\end{proof}

\begin{lemma}\label{le-t}
Let $K$ be a pure $r$-dimensional complex without generalized $(r,t)$-wheels, where $1 \le t \le r$.
Then
\begin{equation}\label{Eq-le-t}
|N^\dn(F,u)|\le t
\end{equation}
for every $r$-face $F\in S_r(K)$ and every vertex $u\in V(K)\setminus F$.
\end{lemma}

\begin{proof}
Assume, to the contrary, that there exists an $r$-face $F \in K$ and a vertex $u \in V(K) \setminus F$ such that $|N^{\dn}(F,u)| \ge t+1$.
Then there exists a $(t+1)$-subset $A \subseteq F$ such that
$$
(F\setminus\{x\})\cup\{u\}\in S_r(K)
$$
for every $x \in A $.
Let $C = F\setminus A$ and $Y=\{u\} \cup A$.
Then the facets $C\cup (Y\setminus\{u\})$ and $ C\cup(Y\setminus\{x\})$, for each $x \in A$, form a copy of
$$
\Delta_C\star \Delta_{t+2}^{t+1}.
$$
Since $|C|=r-t$, and $\Delta_{t+2}^{t+1}$ is a $t$-dimensional basic hole, this is a generalized $(r,t)$-wheel, which yields a contradiction.
\end{proof}

\subsection{\texorpdfstring{$t$}{t}-neighbor uniformity and spectral upper bound}
Now we consider the equality case in \eqref{Eq-le-t}, namely, $t$-neighbor uniformity, and derive a spectral upper bound of a simplicial complex without generalized wheels, related to the equality case of $t$-neighbor uniformity.

\begin{lemma}\label{eig-uni}
Let $K$ be a pure $r$-dimensional complex on $n$ vertices.
If $K$ is $t$-neighbor uniform, then
$$
\q_{r-1}(K)=tn-(t-1)(r+1).
$$
\end{lemma}

\begin{proof}
It is enough to consider $Q_r^\down(K)$.
Since $K$ is $t$-neighbor uniform, for each $r$-face $F$, $|N^{\dn}(F,u)| = t$ for any vertex $u\in V(K)\setminus F$.
Then the number of down neighbors of $F$ is
$$
\sum_{u\in V(K)\setminus F}|N^\dn(F,u)|=t(n-r-1).
$$
By \eqref{Q-dn}, every row sum of the nonnegative matrix $Q_r^\down(K)$ is
$$
(r+1)+t(n-r-1)=tn-(t-1)(r+1).
$$
The all-ones vector is therefore an eigenvector with this eigenvalue.
Since a nonnegative matrix has spectral radius at most its maximum row sum, the spectral radius of $Q_r^\down(K)$ is exactly $tn-(t-1)(r+1)$.
Thus, the same value is $\q_{r-1}(K)$.
\end{proof}

\begin{prop}\label{book}
The generalized $(r,t)$-book $\B_n(r,t)$ is $t$-neighbor uniform, generalized $(r,t)$-wheel-free, and $r$-down path connected.
Consequently,
\[
\q_{r-1}(\B_n(r,t))=tn-(t-1)(r+1).
\]
\end{prop}

\begin{proof}
Write
\[
\B_n(r,t)=\Delta_C\star \Delta_W^{t},
\]
where $|C|=r+1-t$ and $|W|=n-r-1+t$.
Then
\[
S_r(\B_n(r,t))=\{C\cup T:T\in\binom Wt\}.
\]
Let $F=C\cup T$ be an $r$-face and let $u\in W\setminus T$.
A down neighbor of $F$ containing $u$ must be obtained by replacing one 	vertex of $T$ with $u$, since every $r$-face of $\B_n(r,t)$ must contain the whole core $C$.
So, for each $x\in T$, the set
\[
C\cup (T\setminus\{x\})\cup\{u\}
\]
is an $r$-face of $\B_n(r,t)$.
Hence $|N^{\dn}(F,u)|=|T|=t$.
Then $\B_n(r,t)$ is $t$-neighbor uniform.
By Lemma \ref{eig-uni},
\[
\q_{r-1}(\B_n(r,t))=tn-(t-1)(r+1).
\]
	
Next, we show that $\B_n(r,t)$ contains no generalized $(r,t)$-wheels.
Suppose to the contrary that $\Delta_S \star H\subseteq \B_n(r,t)$,
where $|S|=r-t$ and $H$ is a $t$-dimensional basic hole.
Since $|C|=r+1-t>|S|$, choose $c\in C\setminus S$.
For every $t$-face $T$ of $H$, the set $S\cup T$ is an $r$-face of $\B_n(r,t)$.
Therefore $C\subseteq S\cup T$ for every $T\in S_t(H)$.
Since $c\notin S$, we have $c\in T$ for every $t$-face $T$ of $H$.
Thus, every facet of the pure $t$-dimensional complex $H$ contains $c$.
Moreover, if $\tau\in H$ and $c\notin\tau$, then $\tau$ is contained in some facet $T$ of $H$.
Since $c\in T$, we have $\tau\cup\{c\}\in H$.
Hence $H=\{c\}\star H[V(H)\setminus\{c\}]$, so $H$ is a cone with apex $c$.
A cone is contractible, and therefore all its reduced homology groups vanish.
In particular, $\widetilde H_t(H,\R)=0$.
This contradicts the equivalent homological characterization of a basic hole, namely $\beta_t(H,\R)=1$.
Hence $\B_n(r,t)$ is generalized $(r,t)$-wheel free.

Finally, we prove that $\B_n(r,t)$ is $r$-down path connected.
Indeed, let $C\cup T$ and $C\cup T'$ be two $r$-faces.
Write
\[
T\setminus T'=\{x_1,\ldots,x_m\},
\quad
T'\setminus T=\{y_1,\ldots,y_m\}.
\]
Define
\[
T_j=(T\setminus\{x_1,\ldots,x_j\})\cup\{y_1,\ldots,y_j\},
\quad 0\le j\le m.
\]
The two facets $C\cup T_{j-1}$ and $C\cup T_j$ differ in exactly one
external vertex.
Their intersection is
$C\cup (T_{j-1}\cap T_j)$,
which has cardinality $r$ and hence is an $(r-1)$-face.
Thus, consecutive facets in the sequence are down neighbors.
This gives an $r$-down path from $C\cup T$ to $C\cup T'$.
Hence $\B_n(r,t)$ is $r$-down path connected.
\end{proof}

\begin{thm}\label{uni-iff}
Let $K$ be a pure $r$-dimensional complex on $n$ vertices without generalized $(r,t)$-wheels.
Then
\begin{equation}\label{uni-iff-bound}
\q_{r-1}(K)\le tn-(t-1)(r+1).
\end{equation}
Furthermore, if $K$ is $r$-down path connected, then equality in \eqref{uni-iff-bound} holds if and only if $K$ is $t$-neighbor uniform.
\end{thm}

\begin{proof}
The spectral radius $\q_{r-1}(K)$ is the spectral radius of $Q_r^\down(K)$.
Let $f$ be a nonnegative Perron vector of $Q_r^\down(K)$ associated with $\q_{r-1}(K)$.
Choose an $r$-face $F_0$ such that
$$
f(F_0)=\max\{f(F):F\in S_r(K)\}=1.
$$
By Lemma \ref{le-t}, for every $u\in V(K)\setminus F_0$,
\begin{equation}\label{localu}
|N^\dn(F_0,u)|\le t.
\end{equation}
Hence
\begin{equation}\label{local}
|N^\dn(F_0)|=\sum_{u\in V(K)\setminus F_0}|N^\dn(F_0,u)|\le t(n-r-1).
\end{equation}
Using \eqref{Q-dn}, we obtain
\begin{equation}\label{eq-r}
	\begin{split}
		\q_{r-1}(K) & = \q_{r-1}(K) f(F_0) \\
		&= (r+1) f(F_0) + \sum_{F' \in N^\dn(F_0)} f(F')\\
		&\le r+1 + t(n-r-1)= tn - (t-1)(r+1) .
	\end{split}
\end{equation}

Assume now that $K$ is $r$-down path connected and equality holds.
Then equality must hold in every inequality above.
In particular, $f(F')=1$ for every $F'\in N^\dn(F_0)$, and $|N^\dn(F_0,u)|=t$ for every $u\notin F_0$.
Now apply the same argument to any down neighbor $F$ of $F_0$ with $f(F)=1$.
The equality condition forces all down neighbors of $F$ to have Perron coordinate $1$, and also forces $|N^\dn(F,u)|=t$ for every $u\notin F$.
Since $K$ is $r$-down path connected, this propagation along down paths gives $f(F)=1$ for every $F\in S_r(K)$ and
$$
|N^\dn(F,u)|=t
$$
for every $r$-face $F$ and every $u\notin F$.
Thus, $K$ is $t$-neighbor uniform.
The converse is exactly Lemma \ref{eig-uni}.
\end{proof}

We also prove the absolute upper bound stated in Proposition \ref{absolute-bound}.

\begin{proof}[Proof of Proposition \ref{absolute-bound}]
Since $Q_{r-1}^{\up}(K)$ and $Q_r^{\down}(K)$ have the same nonzero eigenvalues,
it is enough to estimate the spectral radius of $Q_r^{\down}(K)$.
Let $F\in S_r(K)$ and let $u\in V(K)\setminus F$.
Every member of $N^\dn(F,u)$ is obtained by replacing one vertex of $F$ by $u$.
Since $|F|=r+1$, we have $|N^\dn(F,u)|\le r+1$.
Hence
\begin{equation}\label{NdF}
	|N^\dn(F)|	=	\sum_{u\in V(K)\setminus F}|N^\dn(F,u)|	\le	(r+1)(n-r-1).
\end{equation}
Let $f$ be a nonnegative Perron vector of $Q_r^{\down}(K)$ corresponding to $\q_{r-1}(K)$, normalized so that
$$f(F_0)=\max\{f(F):F\in S_r(K)\}=1.$$
Then, by \eqref{Q-dn} and \eqref{NdF},
\begin{equation}\label{CM}
\begin{aligned}
	\q_{r-1}(K) & 	= \q_{r-1}(K) f(F_0) \\
	& = (r+1)+\sum_{F'\in N^\dn(F_0)}f(F')\\
	& \le 	(r+1)+(r+1)(n-r-1) \\
    & \le 	(r+1)(n-r).
\end{aligned}
\end{equation}

Now assume that equality holds for a pure $r$-dimensional complex $K$.
Thus equality in \eqref{CM} holds throughout.
In particular, $f(F')=1$ for every $F'\in N^d(F_0)$, and
$|N^d(F_0,u)|=r+1$ for every $u\in V\setminus F_0$.
Equivalently, for every $x\in F_0$ and every $u\in V\setminus F_0$, the set
$(F_0\setminus\{x\})\cup\{u\}$ is an $r$-face of $K$.
We can apply the same argument to any $r$-face $F$ with $f(F)=1$.
Since $F$ also attains the maximum value of the Perron vector, the equality condition again forces
$f(F')=1$ for every $F'\in N^\dn(F)$, and
$|N^\dn(F,u)|=r+1$ for every $u\in V(K)\setminus F$.
	
We use local exchanges to prove that every $(r+1)$-subset of $V(K)$ is an $r$-face of $K$.
Let $T\in\binom{V(K)}{r+1}$.
If $T=F_0$, there is nothing to prove.
Otherwise, write
	\[
	F_0\setminus T=\{x_1,\ldots,x_m\},
	\qquad
	T\setminus F_0=\{u_1,\ldots,u_m\}.
	\]
Starting from $F_0$, define successively
	$F_j=(F_{j-1}\setminus\{x_j\})\cup\{u_j\}$, $1\le j\le m$.
By the propagation statement above, $F_1$ is an $r$-face of $K$ and $f(F_1)=1$.
Applying the same argument successively to $F_1,F_2,\ldots,F_{m-1}$, we obtain
	$F_j\in S_r(K)$ and $f(F_j)=1$ for every $1\le j\le m$.
Since $F_m=T$, we have $T\in S_r(K)$.
Thus $S_r(K)=\binom{V(K)}{r+1}$, and hence $K=\Delta_n^{r+1}$.
	
Conversely, if $K=\Delta_n^{r+1}$, then for every $r$-face $F$ and every $u\in V(K)\setminus F$, all sets $(F\setminus\{x\})\cup\{u\}$ with $x\in F$ are $r$-faces of $K$.
Hence $|N^\dn(F,u)|=r+1$ for every $u\in V(K)\setminus F$.
Therefore every row sum of $Q_r^{\down}(K)$ is
	\[
	(r+1)+(r+1)(n-r-1)=(r+1)(n-r),
	\]
which implies that $\q_{r-1}(\Delta_n^{r+1})=(r+1)(n-r)$ associated with an all-one vector.
\end{proof}

\section{Structure of \texorpdfstring{$t$}{t}-neighbor uniform complexes}
The equality case of the spectral bound in Theorem \ref{uni-iff} forces the local $t$-neighbor uniform condition.
The purpose of this section is to prove that, under the link-homology obstruction (or forbidding generalized wheels), for sufficiently large $n$, local exchange uniformity forces all facets to share a global common core.

Throughout this section, let $K$ be a pure $r$-dimensional complex on vertex set $V$.
Assume that $K$ is $r$-down path connected, generalized $(r,t)$-wheel-free, and $t$-neighbor uniform, and
\begin{equation}\label{n-large}
n\ge r+2+\binom{r+1}{t}\binom{r}{t}.
\end{equation}

\subsection{Co-graph complexes}

Let $G=(V,E)$ be a graph with $|V|=r+3$.
The \emph{co-graph complex associated with $G$}, denoted by $K_G$, is the $r$-dimensional pure complex on the vertex set $V$ whose $r$-faces are the complements of the edges of $G$:
$$
S_r(K_G)=\{V\setminus e:e\in E(G)\}.
$$
For an edge $e\in E(G)$, write $F_e:=V\setminus e$.
Then two $r$-faces $F_e$ and $F_{e'}$ are down neighbors if and only if $|e\cap e'|=1$.

\begin{lemma}\label{regular}
Let $V$ be a vertex set of size $r+3$.
Let $G$ be a graph on $V$ with at least one edge, and let $K_G$ be its co-graph complex.
For $0\le t\le r+1$, the complex $K_G$ is $t$-neighbor uniform if and only if the induced subgraph of $G$ on its non-isolated vertices is $(t+1)$-regular.
In this case, the number $i(G)$ of isolated vertices satisfies $0\le i(G)\le r+1-t$.
\end{lemma}

\begin{proof}
Let $e=\{u,v\}\in E(G)$, and let $F_e=V\setminus e$.
Consider the outside vertex $u\notin F_e$.
An $r$-face $F_{e'}$ belongs to $N^\dn(F_e,u)$ if and only if $F_{e'}$ is a down neighbor of $F_e$ and $u\in F_{e'}$.
Equivalently, $|e\cap e'|=1$ and $u\notin e'$.
Since $e=\{u,v\}$, this means that $e'$ is incident with $v$ and is different from $e$.
Therefore
$$
|N^\dn(F_e,u)|=\deg_G(v)-1.
$$
Similarly,
$$
|N^\dn(F_e,v)|=\deg_G(u)-1.
$$
Thus $K_G$ is $t$-neighbor uniform if and only if every endpoint of every edge of $G$ has degree $t+1$.
This is exactly the statement that the induced subgraph on the non-isolated vertices is $(t+1)$-regular.
If such a non-isolated vertex exists, it has $t+1$ neighbors; hence, the number of non-isolated vertices is at least $t+2$.
Since $|V|=r+3$, we have $i(G)\le r+1-t$.
\end{proof}

\begin{rmk}
Suppose the  complex $K_G$ in Lemma \ref{regular} is $t$-neighbor uniform.
The induced subgraph of $G$ on non-isolated vertices, denoted by $G^\circ$, is $(t+1)$-regular, but is not necessarily connected.
For example, $G^\circ$ consists of at least two $(t+2)$-cliques.
In this case, $r \ge 2t+1$.
So, if $r \le 2t$, then $G^\circ$ is connected.
From the proof of Lemma \ref{regular}, $G^\circ$ is connected if and only if $K_G$ is $r$-down path connected, equivalently, any two edges of $G$ are connected by a path.
\end{rmk}

\subsection{Local exchanges}

For an $r$-face $F\in S_r(K)$ and a vertex $u\in V\setminus F$, define the \emph{exchange set}
$$
A_F(u):=\{x\in F:(F\setminus\{x\})\cup\{u\}\in S_r(K)\},
$$
and the corresponding \emph{forbidden set}
$$
C_F(u):=F\setminus A_F(u).
$$
Since $K$ is $t$-neighbor uniform, $|A_F(u)|=t$ and $|C_F(u)|=r+1-t$.
The next lemma shows that $A_F(u)$ does not depend on $u$.

\begin{lemma}\label{ex-same}
For every $r$-face $F$ of $K$, there exists a fixed $t$-subset $A_F\subseteq F$ such that
$$
A_F(u)=A_F
$$
for every $u\in V\setminus F$.
\end{lemma}
\begin{proof}
Fix an $r$-face $F\in S_r(K)$.
For every $u\in V\setminus F$, the set $A_F(u)$ is a $t$-subset of $F$.
There are only $\binom{r+1}{t}$ possible such subsets.
By the lower bound on $n$ in \eqref{n-large},
$$
|V\setminus F|=n-r-1\ge \binom{r+1}{t}\binom rt+1.
$$
Hence, by the pigeonhole principle, there exists a $t$-subset $A\subseteq F$ and a set $U\subseteq V\setminus F$ with
$$
|U|\ge \binom rt+1
$$
such that $A_F(x)=A$ for all $x\in U$.

Suppose to the contrary that there exists a $w \in V \setminus F$ such that $A_{F}(w) \neq A$.
For convenience, denote $A_{F}(w) := A_w$.
By assumption, $w \notin U$ and $A \setminus A_w$ is a non-empty set.
Hence, there exists at least one vertex $s \in A \setminus A_w$.

For any vertex $x \in U$,
let $V_{w,x} := F\cup \{ w,x\}$, and let $G_{w,x}$ be the graph on vertex set $V_{w,x}$ such that $K[ V_{w,x} ] $ is the co-graph complex associated with the graph $G_{w,x}$, namely, $e \in E(G_{w,x})$ if and only if $V_{w,x} \setminus e \in S_r(K[V_{w,x}])$.
We claim that $K[V_{w,x}]$ is $t$-neighbor uniform.
Indeed, let $F'\in S_r(K[V_{w,x}])$ and let
$u'\in V_{w,x}\setminus F'$.
For every $a\in F'$, the possible exchange face
$(F'\setminus\{a\})\cup\{u'\}$
is contained in $V_{w,x}$.
Since $K[V_{w,x}]$ is induced, this exchange face belongs to
$K[V_{w,x}]$ if and only if it belongs to $K$.
Therefore the number of valid exchanges in $K[V_{w,x}]$ equals the number of valid exchanges in $K$.
Since $K$ is $t$-neighbor uniform, $K[V_{w,x}]$ is also $t$-neighbor uniform.

By Lemma \ref{regular},
the induced subgraph of $G_{w, x}$ on its non-isolated vertices must be $(t+1)$-regular.
Since $s \in A \subseteq F$ and $s\notin A_w$, we have $F\setminus \{s\} \cup \{x\} \in S_r(K[ V_{w,x} ])$ and $F\setminus \{s\} \cup \{w\} \notin S_r(K[ V_{w,x} ])$, i.e.,
$\{s,w\}\in E(G_{w,x})$ and $\{s,x\}\notin E(G_{w,x})$.
Thus $s$ is not isolated in $G_{w,x}$.
Since the non-isolated part of $G_{w,x}$ is $(t+1)$-regular, the vertex
$s$ has exactly $t+1$ neighbors.
One of these neighbors is $w$, while $x$ is not a neighbor of $s$.
Hence the remaining $t$ neighbors of $s$ all lie in $F\setminus\{s\}$.
Denote this $t$-set by $P_{s,x}$.
Then
\[
N_{G_{w,x}}(s)=\{w\}\cup P_{s,x}.
\]
There are at most $\binom{r}{t}$ possible choices for $P_{s,x}$.
Since $	|U| \ge \binom{r}{t} +1$, the pigeonhole principle guarantees the existence of two distinct vertices $x, y\in U$ such that $P_{s,x} = P_{s,y}$.
For convenience, let $P := P_{s, x} = P_{s, y} $, and let	$S= F \setminus (\{s\} \cup P)$.

From the preceding discussion, we have $ N_{G_{w,x}}(s)=N_{G_{w,y}}(s)$; let us denote this common neighborhood by $N(s)$.
Thus, $|P| = t$, $|N(s)| = t+1$, $|S| = r-t$.
For every vertex $a \in N(s)$, the presence of the edge $sa\in E(G_{w,x})$ implies
$$
V_{w,x} \setminus \{s,a\} = S\cup\{x\}\cup(N(s) \setminus\{a\}) \in S_r(K).
$$
Similarly, 	applying the same reasoning to the edge  $sa\in E(G_{w,y})$ gives
\[
V_{w,y}\setminus\{s,a\}
=
S\cup\{y\}\cup (N(s)\setminus\{a\})
\in S_r(K).
\]
Since $K$ is closed under taking subsets, all sets
$$\{x\}\cup(N(s)\setminus\{a\}), \quad \{y\}\cup(N(s)\setminus\{a\})$$
are $t$-faces of $K$ for every $a\in N(s)$.
Let $H$ be the $t$-dimensional complex generated by these $2(t+1)$ $t$-faces.
Consequently, $H \cong \Delta_2^{1} \star \Delta_{t+1}^{t} \cong \diamondsuit_{t+3}^{t+1}$.
Thus $H$ is a $t$-dimensional basic hole.
Moreover, for every facet $T$ of $H$, the preceding construction gives
$S\cup T\in S_r(K)$.
Since $K$ is a simplicial complex, it follows that
\[
\Delta_S\star H\subseteq K.
\]
As $|S|=r-t$ and $H$ is a $t$-dimensional basic hole,
$\Delta_S\star H$ is a generalized $(r,t)$-wheel contained in $K$,
contrary to the assumption that $K$ is generalized $(r,t)$-wheel free.
Hence $A_F(u)=A$ for every $u\in V\setminus F$.
Taking $A_F:=A$ proves the lemma.
\end{proof}

\subsection{The global common core}

For each $r$-face $F$, let $A_F$ be the fixed exchange set given by Lemma \ref{ex-same}, and define
$$
C_F:=F\setminus A_F.
$$
Thus $|C_F|=r+1-t$.

\begin{lemma}\label{core}
There exists a fixed subset $C\subseteq V$ with $|C|=r+1-t$ such that every $r$-face of $K$ contains $C$.
\end{lemma}

\begin{proof}
	By Lemma \ref{ex-same}, for every $r$-face $F\in S_r(K)$, there exists a fixed $t$-subset $A_F\subseteq F$ such that $A_F(u)=A_F	$ for every $u\in V\setminus F$. Define $C_F:=F\setminus A_F$.
Then $|C_F|=r+1-t$.
Fix an $r$-face $F_0\in S_r(K)$, and set $	C:=C_{F_0}$.
We prove that $C_F=C$ for every $F\in S_r(K)$.

It is enough to prove that this equality propagates along one down-neighbor step.
Let $F\in S_r(K)$ satisfy $C_F=C$,
and let $F'$ be a down neighbor of $F$.
Write $ F'=(F\setminus\{x\})\cup\{u\}$,
where $x\in F$ and $u\in V\setminus F$.
Since $F'\in S_r(K)$, by the definition of the exchange set we have $x\in A_F(u)=A_F$.
Hence $ x\notin C_F=C$.
Since $C\subseteq F$ and $x\notin C$, it follows that $ C\subseteq F'$.

We claim that $C_{F'}=C$.
Suppose, to the contrary, that	$C_{F'}\ne C$.
Since both $C_{F'}$ and $C$ have cardinality $r+1-t$, and both are contained in $F'$, there exists	$y\in C\setminus C_{F'}$.
Moreover, $x\notin F'$.
By Lemma \ref{ex-same} applied to the facet $F'$, the set $C_{F'}$ is independent of the outside vertex.
Therefore $ C_{F'}=F'\setminus A_{F'}(x)$.
Since $y\notin C_{F'}$, we have	$y\in A_{F'}(x)$.
Thus
$$(F'\setminus\{y\})\cup\{x\} = 	(F\setminus\{y\})\cup\{u\}\in S_r(K).$$
This implies $y\in A_F(u)$.
On the other hand, since $y\in C=C_F=F\setminus A_F$, we have $y\notin A_F=A_F(u)$,
a contradiction.
Hence $C_{F'}=C.$

Now let $G\in S_r(K)$ be arbitrary.
Since $K$ is $r$-down path connected, there exists a sequence of $r$-faces $F_0,F_1,\ldots,F_m=G$
such that $F_{i+1}$ is a down neighbor of $F_i$ for every $0\le i \le m-1$.
We have just shown that if	$C_{F_i}=C$ then $C_{F_{i+1}}=C$.
Since $C_{F_0}=C$, induction gives $C_G=C$.
Therefore every $r$-face $G$ contains the fixed set $C$.
\end{proof}

\subsection{Completion to the generalized book}

\begin{lemma}\label{ex-complete}
Let $C$ be the fixed core from Lemma \ref{core}, and put $W:=V\setminus C$.
Then
$$
S_r(K)=\{C\cup T:T\in\binom Wt\}.
$$
\end{lemma}

\begin{proof}
By Lemma \ref{core}, every $r$-face of $K$ contains $C$.
Hence every $r$-face of $K$ is of the form $C\cup T$ for some $T\in\binom{W}{t}$.
Let $\mathcal{T}$ be the family of all valid external $t$-subsets that can form a facet with $C$:
$$
\mathcal{T}=
\{T \in \binom{W}{t}: C \cup T\in S_r(K)\}.
$$

Since $K$ is pure $r$-dimensional, $\mathcal T$ is nonempty.
We shall prove that $\mathcal T$ is closed under local exchanges.
Let $T\in\mathcal T$ and let $T'=(T\setminus\{x\})\cup\{u\}$, where $x\in T$ and $u\in W \setminus T$.
Set $F:=C\cup T$.
Then $F\in S_r(K)$ and $u\notin F$.	
We first observe that $A_F(u)=T$.
Indeed, if $c\in C$, then $(F\setminus\{c\})\cup\{u\}$ does not contain $C$, and therefore it cannot be an $r$-face of $K$ by Lemma \ref{core}.
Thus $A_F(u)\cap C=\emptyset$, and hence $A_F(u)\subseteq F\setminus C=T$.
Since $K$ is $t$-neighbor uniform, $|A_F(u)|=t$.
Since $|T|=t$, we obtain $A_F(u)=T$.
Therefore
$$(F\setminus\{x\})\cup\{u\} =C\cup(T\setminus\{x\})\cup\{u\}=C\cup T'\in S_r(K).$$
Hence $T'\in\mathcal T$.

We now use local exchanges to show $\mathcal T=\binom{W}{t}$.
Let $T_0\in \mathcal T$ and let $T \in\binom{W}{t}$ be arbitrary.
If $T_0=T$, there is nothing to prove.
Otherwise, choose an element $x\in T_0\setminus T$ and an element
$u\in T\setminus T_0$.
By the exchange closure proved above,
\[
(T_0\setminus\{x\})\cup\{u\}\in\mathcal T.
\]
This operation increases the size of the intersection with $T$ by one.
Repeating the argument finitely many times, we obtain $T\in\mathcal T$.
Since $T$ was arbitrary, $\mathcal T=\binom{W}{t}$.
Therefore
\[
S_r(K)=\{C\cup T:T\in\binom Wt\}.
\]
\end{proof}

Now we can prove the main result of this paper.
By Lemma \ref{link}, Theorem \ref{main} is equivalent to Corollary \ref{wheel-free-main}.
So we only prove Corollary \ref{wheel-free-main} here.

\begin{proof}[Proof of Corollary \ref{wheel-free-main}]
By Theorem \ref{uni-iff}, we get the desired upper bound
$$
\q_{r-1}(K)\le tn-(t-1)(r+1).
$$

Assume now that $K$ is $r$-down path connected and that \eqref{n-large} holds.
If equality holds, then Theorem \ref{uni-iff} implies that $K$ is $t$-neighbor uniform.
By Lemma \ref{core}, there exists a fixed subset $C\subseteq V(K)$ with $|C|=r+1-t$ such that every $r$-face of $K$ contains $C$.
Let $W:=V(K)\setminus C$.
Then $|W|=n-r-1+t$.
By Lemma \ref{ex-complete},
$$
S_r(K)=\left\{C\cup T:T\in\binom{W}{t}\right\}.
$$
Therefore
$$
K\cong \Delta_{r+1-t} \star \Delta_{n-r-1+t}^{t} =\B_n(r,t).
$$

Conversely, suppose $K\cong \B_n(r,t)$.
By Proposition \ref{book}, $\B_n(r,t)$ is $r$-down path connected, generalized $(r,t)$-wheel-free and
$\q_{r-1}(\B_n(r,t))=tn-(t-1)(r+1)$.
The sufficiency follows.
\end{proof}

\end{document}